\documentclass[a4paper,10pt]{article}
\usepackage[utf8]{inputenc}
\usepackage{amsmath}
\usepackage{todonotes}
\usepackage{amsthm}
\usepackage{amssymb}
\usepackage{hyperref}
\usepackage{pdfsync}
\usepackage{mathrsfs}
\usepackage{textcomp}
\usepackage{xcolor}
\usepackage{ulem}
\usepackage{authblk}
\usepackage{graphicx}

\newtheorem{lem}{Lemma}[section] 

\newtheorem{thm}{Theorem}[section]

\newtheorem{cor}[thm]{Corollary}
\numberwithin{equation}{section}
\definecolor{forestgreen}{rgb}{0.0, 0.27, 0.13}

\theoremstyle{definition}
\newtheorem{defn}{Definition}[section]

\theoremstyle{remark}
\newtheorem{rmk}{Remark}[section]

\title{Stability of multi-population traffic flows}
\date{\empty}
\author[a]{Amaury Hayat}
\author[b]{Benedetto Piccoli}
\author[c]{Shengquan Xiang}
\affil[a]{CERMICS, Ecole des Ponts ParisTech, Marne-la-Vallée, France}
\affil[b]{Department of Mathematical Sciences and Center for Computational and Integrative Biology, Rutgers University–Camden, 303 Cooper St, Camden, NJ, USA}
\affil[c]{Bâtiment des Mathématiques, EPFL, Station 8, CH-1015 Lausanne, Switzerland.}

\begin{document}

\maketitle
\begin{abstract}
    Traffic waves, known also as stop-and-go waves or phantom hams, appear naturally as traffic instabilities, also in confined environments as a ring-road.
    A multi-population traffic is studied on a ring-road, comprised of drivers with stable and unstable behavior.
    There exists a critical penetration rate of stable vehicles above which the system is stable, and under which the system is unstable. In the latter case, stop-and-go waves appear, provided enough cars are on the road. The critical penetration rate is explicitly computable, and, in reasonable situations, a small minority of aggressive drivers is enough to destabilize an otherwise very stable flow. 
    This is a source of instability that a single population model would not be able to explain.
    Also, the multi-population system can be stable below the critical penetration rate if the number of cars is sufficiently small. Instability emerges as the number of cars  increases, even if the traffic density remains the same (i.e. number of cars and road size increase similarly). This shows that small experiments could lead to deducing imprecise stability conditions.
\end{abstract}

\section{Introduction}

Traffic jams that appear for seemingly no reason are a phenomenon that everyone has experienced.  Even though there are no accidents, no lane reductions, no highway exits, so called stop-and-go waves form.\\
This paradox has a mathematical answer: under certain conditions, traffic equilibrium states are unstable. This phenomenon has been observed and studied in many works. In particular the study of \cite{Sugiyama}, and the experiment  associated, provided evidence that such waves form on a circular road with only about twenty drivers starting from same spacing and speed. The impact of the waves include increased fuel consumption as measured in a similar experiment
\cite{stern2018dissipation}.

The consequence of such traffic instabilities are important: accelerating and decelerating increase strongly the fuel consumption and the gas emissions compared to the associated equilibrium flow. 
Reducing these instabilities by acting on the traffic, for instance rendering the equilibrium stable, would have a strong impact. This is why many approaches have been considered to solve the problem, from ramp-metering control \cite{Rampmeter1,Rampmeter2,BastinCoron1D, gugat2005optimal,yu2019traffic,yu2020bilateral,Aurioltraffic} to the use of connected autonomous vehicles which \textcolor{black}{act as ``wave-dampeners''} in the traffic (see for instance \cite{talebpour2016influence,wang2015cooperative,CuiSeibold2017,stern2018dissipation,zheng2020smoothing,wang2020controllability} and \cite{delleLiardmicro} for a more detailed review). But, to be able to act on traffic efficiently, one needs to understand how these traffic waves emerge and their stability.

The main trigger for the creation of waves is the collective behavior of human agents on the road, which typically becomes unstable when the equilibrium velocity cross a certain threshold and becomes too low.
They have been extensively studied in the literature, see \cite{FKNRS09,Helb01,KK94}, and recently explained theoretically in \cite{CuiSeibold2017} in a ring-road framework similar to \cite{Sugiyama,stern2018dissipation}. These works study a single-phase traffic where all drivers follow the same car-following model. However, in real life, it is likely that the instability of equilibrium flows is strongly influenced by the differences in driving characteristics among agents. 
These differences could occur due to the type of vehicle (trucks, SUV, small cars, etc.) or differences in behavior between drivers. For instance, trucks platooning was studied for its fuel-consumption impact in mixed traffic \cite{trucks,trucks2}. Another interesting situation is when some drivers adopt a collaborative way of driving \cite{HALLE2005320}. In this case the traffic can be seen as a system with two populations: the standard drivers and those who adopt the collaborative behavior. Several interesting questions could be asked: 
\textcolor{black}{
\begin{itemize}
    \item Can the presence of two or several populations create some instabilities?
    \item Would a minority of collaborative drivers be able to render stable a traffic that would otherwise be unstable?
On the contrary, could a minority of aggressive drivers make an otherwise very stable traffic unstable?
\end{itemize} 
}

In this paper we provide some answers to these questions.
Specifically, a mixed traffic on a ring-road is analyzed,
where two or more populations coexists. After characterizing the equilibrium flows, \textcolor{black}{we study the stability of the overall flow depending on the proportion of cars in each class of vehicles.}
\textcolor{black}{W}e show that, when an instable and a stable population co-exist, there is a critical penetration rate above which the system is stable and under which the system is unstable, provided there is enough cars on the road. This penetration rate is explicitly computable. We also provide qualitative bounds on the critical penetration rate $\tau_{0}$, in order to elucidate what makes a collaborating behavior effective from a traffic stability point of view.\\
\textcolor{black}{For reasonable parameters,} the critical penetration rate of stable vehicle is very high. This means that, not taking into account the differences in dynamics, can lead to not 
\textcolor{black}{understanding}
the source of instability of the flow. Indeed, we show that 
a small minority of aggressive drivers, in an otherwise very stable flow, is enough to break the stability. This is something that a single-population model cannot grasp.\\
Surprisingly, we also show that the ability of a group of drivers to stabilize or destabilize the system only depends on its penetration rate in the traffic and not on the order of the cars.
One could think for instance that three trucks in a row is less effective to stabilize the traffic than three trucks equally distributed on the ring-road. But it is not. \\
Finally, we also show that small experiments, as considered in \cite{Sugiyama,stern2018dissipation}, could lead to underestimate the instability of the traffic flow and the critical penetration rate. More precisely, for a small number of vehicles the system can still be stable even below the critical penetration rate, but becomes unstable when the number of cars increases.

The paper is organized as follows: in Section \ref{sec:traffic} we present the framework and the system, in Section \ref{sec:main-results} we state the main results, in Section \ref{sec:analysis} we provide an analysis of the traffic around its equilibrium flows, while in Section \ref{sec2pase} we show our results for a two-phase traffic flow and in Section \ref{sec:multi} for a general multi-phase traffic flow. Finally, in Section \ref{sec:numerics} we provide some numerical simulations to illustrate our results.

\section{General traffic model on a ring-road}
\label{sec:traffic}
\textcolor{black}{We study a general traffic model with $n$ vehicles and a ring-road of} length $L$. Mathematically this means that we consider the system on the domain $\mathbb{T}= \mathbb{R}/L \mathbb{Z}$.  We denote by $\{x_j(t)\}_{j=1}^n$ the location of  the cars. By further denoting the headway and velocity of the cars as
\begin{equation}
h_j(t)= x_{j+1}(t)- x_j(t) \textrm{ and } v_j(t)= \dot{x}_j(t),
\end{equation}
the traffic is generally described by {\color{black}
\begin{equation}\label{model-ge}
\begin{cases}
     \dot{h}_j(t)= v_{j+1}(t)- v_j(t), \\
    \dot v_j(t)= f_j(h_j(t), \dot h_j(t), v_j(t)),
\end{cases} 
\, \forall j\in \{1, 2,..., n\}
\end{equation}
\textcolor{black}{where $f_{j}$ is the car following model for the driver $j$ and} with the convention $x_{n+1}= x_1+ L$ on the ring road (or, equivalently, $x_{n+1}= x_{\textcolor{black}{1}}$ in $\mathbb{T}$),} \textcolor{black}{which means that 
\begin{equation}
\label{eq:cons}
    \sum\limits_{j=1}^{n}h_{j}(t) = L,\;\;\forall\; t\geq 0.
\end{equation}
} For a general car following model we usually have the following physical conditions on $f_j(h, 0, v)$:
\begin{equation}\label{nat-con-f}
    \frac{\partial}{\partial h} f_j(h, 0, v)>0, \;  \frac{\partial}{\partial v} f_j(h, 0, v)<0,\;\textcolor{black}{\frac{\partial}{\partial \dot h} f_j(h, 0, v)>0.}
\end{equation}
\textcolor{black}{The first condition simply means that, for a given speed, the incentive to accelerate increase with the headway. The second condition means that for a given headway the incentive to accelerate decrease with the speed. And the third condition means that} for a given headway and speed, the incentive to accelerate increases if the headway is \textcolor{black}{currently} increasing (i.e. if the leading vehicle is moving away). These conditions can be found in nearly all car following models \textcolor{black}{(e.g. Intelligent Driver Models \cite{treiber2000congested}, Follow-the-Leader \cite{gazis1961nonlinear}, Bando-FTL \cite{bando1995dynamical, CuiSeibold2017}, etc.)}\\

\textcolor{black}{While, in most traffic analysis, the car following model $f$ is chosen to be identical for all cars on the road, in general the car-following model $f_j$ depends}
on the driving habit of the $j$-th driver, which may differ from driver to driver. It can also depends on the  $j$-th vehicle itself since different kinds of cars may result in different parameters ($e.g.$ one can simply compare a truck and a mini cooper on the road). We summarize \textcolor{black}{three typical} cases below:
\begin{itemize}
    \item {\it Unified model}
    
    This is the case that is most commonly considered, where $f_j$ does not depend on $j\in \{1, 2,..., n\}$ and there is a single type of vehicle on the road. The stability of this system has been studied in \cite{CuiSeibold2017} and their results are recalled below.  Under such a unified setting, a further special case can be the so-called Bando-Follow the leader model (Bando-FTL, or equivalently OV-FTL), where
\begin{equation}\label{Bando-FTL}
    f_j(h, \dot{h}, v)= a \cdot \left(V(h)- v\right)+ b \cdot\frac{\dot{h}}{h^2}.
\end{equation}
This model was studied for instance in \cite{HPT2021,CuiSeibold2017,delleLiardmicro}. It combines a Bando (or Optimal Velocity) part introduced in \cite{bando1995dynamical} which represents the preference of a driver to reach its ``own" optimal velocity (that depends on the headway), and a ``Follow the Leader" (FTL) part introduced in \cite{gazis1961nonlinear} which represent the incentive of the driver to mimic its leader. In \eqref{Bando-FTL} the Bando part has a weight $a$ while the FTL part has a weight $b$.

    \item {\it Mixed traffic} or {\it Collaborative driving}
    
    In a mixed traffic, the functions $\{f_j\}_{j=1}^n$ are chosen from a finite set. In other words, the drivers and the vehicles can be classified in finitely many categories: 
    \begin{equation}\label{ex:cd}
        f_j\in \{F_k: k=1,2,..., m\}, \; \forall j\in \{1, 2,...,n\}
    \end{equation}
    with $m\ll n$. \textcolor{black}{A particular example is the two population traffic, where $m=2$ while $n$ might be large. This situation corresponds for instance to a collaborative driving where some drivers follow a collaborative behaviors and have a ``good" function $F_{1}$ while the rest of the traffic follows a standard function $F_{2}$ that would lead to instabilities and stop-and-go waves. We deal with this case in Section \ref{sec2pase}.} To illustrate the mixed traffic setting, we can look at the case where the $\{f_{j}\}$ are given by the {\color{black} ``general" } Bando-FTL model,  characterised by
    \begin{gather}
    f_j(h, \dot{h}, v)= a_j \cdot \left(V_j(h)- v\right)+ b_j \cdot\frac{\dot{h}}{h^2}, \label{ex:cdB} \\
    (a_j, b_j, V_j)\in \{(A_k, B_k, \mathcal{V}_k): k=1, 2,..., m\}, \label{ex:cdB2}
    \end{gather}
    In this case, \textcolor{black}{the weights} $(a_j, b_j)$ represents driving habit of the driver, and that the ``Bando function" \textcolor{black}{$V_j(h)$ represents its velocity preference, which can depend for instance of the type of vehicle (cars, trucks, etc.). {\color{black}We remark that in the literature  for Bando model the fucntion $V(h)$ is fixed, however, in reality different type of vehicles may have direct influence, that is the reason we call it ``general" Bando and adapt different $V_j(h)$ fucntions.}  We provide in Section \ref{sec:numerics} numerical simulations for this particular model.}  
    
    \item {\itshape Mixed traffic with common velocity preference}
    
    \textcolor{black}{This is a particular case of mixed traffic where
    all drivers have the same velocity preference. For instance for the Bando-FTL model this means that $V_{j}=V$ and is the same across drivers, while $a_{j}$ and $b_{j}$ remain driver-dependent. As we will see in Section \ref{sec:equi}, in this case the equilibrium flows are the same as the equilibrium flows in the {\itshape unified model}.} 
\end{itemize}
\section{Main results}
\label{sec:main-results}
We are interested in the stability of the system \eqref{model-ge} around its equilibrium flow. A precise description of the equilibrium flows is given below in Section \ref{sec:equi}. In the unified model, where all drivers have the same car-following model $f_{j} = f$, the stability of the system was studied in \cite{CuiSeibold2017}. Denoting $(h_{eq},v_{eq})$ the uniform flow equilibrium and 
\begin{equation}\label{def-abg}
    (\alpha, \beta, \gamma)= \left(\frac{\partial f}{\partial h}, \;\frac{\partial f}{\partial \dot{h}}- \frac{\partial f}{\partial v},\; \frac{\partial f}{\partial \dot{h}}\right)(h_{eq}, 0, v_{eq}),
\end{equation}
 satisfying the natural ``common sense" condition:
\begin{equation}\label{eq-con-com-se}
\alpha>0, \;   \beta>\gamma>0,
\end{equation}
the authors of \cite{CuiSeibold2017} showed the following

\begin{thm}[\cite{CuiSeibold2017}, {\it unified model}]
\label{lem1}
 The uniform flow equilibrium $(h_{eq},v_{eq})$ of the system \eqref{model-ge} is \begin{itemize}
    \item {\it locally stable} around this flow, if \begin{equation}\label{eq:pos}
    \beta^2-\gamma^2-2\alpha\geq 0,
\end{equation}
    \item {\it unstable} around this flow provided $n$ sufficiently large, if
     \begin{equation}\label{eq:neg}
    \beta^2-\gamma^2-2\alpha< 0.
\end{equation}
\end{itemize}
\end{thm}
In this paper we investigate what happens when there is not anymore a single population of vehicles but several. In particular, what happens when vehicles with an unstable behavior (i.e. which satisfies \eqref{eq:neg}) coexists with vehicles with a stable behavior (i.e. which satisfies \eqref{eq:pos})?\\

Our main results are the following: consider first a two population system where the vehicles follow either the car following model $f^{1}$ or the car following model $f^{2}$. In this case, any stationary state (or equilibrium flow) can be described by $(\bar{h}_{i},\bar{v})$ with $i\in\{1,2\}$ (see Section \ref{sec:equi}). Define 
\begin{gather}
(\alpha^{i},\beta^{i}, \gamma^{i}) =\left(\frac{\partial f^{i}}{\partial h}, \;\frac{\partial f^{i}}{\partial \dot{h}}- \frac{\partial f^{i}}{\partial v},\; \frac{\partial f^{i}}{\partial \dot{h}}\right)(\bar{h}_{i}, 0, \bar{v}),\;\;i\in\{1,2\},
\end{gather}
we have the following theorem when both populations have a stable behavior
\begin{thm}
\label{th1}
If $(\beta^{1})^2-(\gamma^{1})^2-2\alpha^{1}\geq 0$ and $(\beta^{2})^2-(\gamma^{2})^2-2\alpha^{2}\geq 0$, then the steady-state $(\bar{h}_{i},\bar{v})_{i\in\{1,2\}}$ of the ring road system \eqref{model-ge}--\eqref{eq:cons} is locally exponentially stable.
\end{thm}
When the population have different behaviors, we denote $n_{1}$ and $n_{2}$ the number of cars of each population in the road, and we have the following theorem
\begin{thm}
\label{th12}
Assume that $\Delta^{1}:=(\beta^{1})^2-(\gamma^{1})^2-2\alpha^{1}> 0$ and $\Delta^{2}:=(\beta^{2})^2-(\gamma^{2})^2-2\alpha^{2}<0$.
      There exists a  critical penetration rate $\tau_0\in (0, 1)$ such that for any pair $(n_1, n_2)\in \mathbb{N}^2$ verifying
\begin{equation}
    \frac{n_1}{n_1+ n_2} >  \tau_0, 
\end{equation}
the  ring road traffic system \eqref{model-ge}--\eqref{eq:cons} is locally exponentially stable around any the equilibrium flow $(\bar{h}_{i},\bar{v})_{i\in\{1,2\}}$, whatever the ordering of the cars.

On the other hand, for any fixed penetration rate
\begin{equation}
\tau= \frac{n_{1}}{n_{1}+n_{2}}    <\tau_{0}
\end{equation}
there exists some $M>0$ effectively computable such that for any ${n_1, n_2}\in \mathbb{N}^2$ satisfying
$n_1+n_2>M$,
the ring road traffic system \eqref{model-ge}--\eqref{eq:cons} is unstable around the the equilibrium flow $(\bar{h}_{i},\bar{v})_{i\in\{1,2\}}$.

Moreover, the critical penetration rate $\tau_{0}$ is explicitly given by
\begin{equation}
\label{eq:deftau0}
\begin{split}
    \tau_{0} &=1- \left(1+\max\left\{-\frac{H_2(y)}{H_1(y)};\;  y\in \left(0, \Gamma^2\right]\right\} \right)^{-1},\\
      \text{ where }\;\;\;H_{i}(y):&=\log\left(\frac{(\alpha^{i})^2+ (\gamma^{i})^2 y}{(\alpha^{i})^2+ ((\beta^{i})^2- 2\alpha^{i}) y+ y^2}\right),\;\;\text{ for }i\in\{1,2\},\\
      \text{ and }\;\;\; &\Gamma^2:= \frac{-(\alpha^2)^2+  \sqrt{(\alpha^2)^4- (\alpha^2)^2(\gamma^2)^2 \Delta^2}}{(\gamma^2)^2}\in (0, -\Delta^2).\\
      \end{split}
\end{equation}
\end{thm}

Even though $\tau_{0}$ can be computed easily through a minimization algorithm, we can also give some practical upper and lower bounds for qualitative studies:
\begin{equation}
\begin{split}
\tau_{0} &\geq \frac{-\Delta^2 (\alpha^1)^2}{\Delta^1 (\alpha^2)^2-\Delta^2 (\alpha^1)^2},\\
\text{ and }\;\;\; \tau_{0}&\leq  \frac{(-\Delta^2) \left((\alpha^1)^2+ (\gamma^1)^2 \Gamma^2\right)\left((\alpha^1)^2+ ((\beta^1)^2- 2\alpha^1) \Gamma^2+ (\Gamma^2)^2\right)}{(\beta^2)^2 (\alpha^1)^2 \Delta^1+ (-\Delta^2) \left((\alpha^1)^2+ (\gamma^1)^2 \Gamma^2\right)\left((\alpha^1)^2+ ((\beta^1)^2- 2\alpha^1) \Gamma^2+ (\Gamma^2)^2\right)}.
\end{split}
\end{equation}
This allows a remark: for a stable class of vehicles to be efficient at stabilizing a mixed traffic flow with a small penetration rate, $\alpha^{1}$ should be small. This means that an efficient collaborating behavior for stabilizing traffic flow will be composed of vehicles driving without taking much the headway into consideration (apart for safety reasons).\\

Finally, when one of the population has a stable behavior but corresponding to the critical case $(\beta^{1})^2-(\gamma^{1})^2-2\alpha^{1}=0$ and the other population has an unstable behavior, we have the following theorem
\begin{thm}
\label{th13}
If $\Delta_{1}:=(\beta^{1})^2-(\gamma^{1})^2-2\alpha^{1}=0$ and $\Delta_{2}:=(\beta^{2})^2-(\gamma^{2})^2-2\alpha^{2}<0$, then provided sufficiently many vehicles on the ring road, the traffic system \eqref{model-ge}--\eqref{eq:cons} is unstable around the the equilibrium flow $(\bar{h}_{i},\bar{v})_{i\in\{1,2\}}$.
\end{thm}

\textcolor{black}{This can be generalized for an multi-phase traffic with more than two populations. (see Theorem \ref{thm:61} in Section \ref{sec:multi})} 

\section{Analysis of the system}
\label{sec:analysis}

\subsection{
\textcolor{black}{Characterization of the equilibrium flow}}
\label{sec:equi}
The aim of this section is to describe the stationary states, or equilibrium flows, of the ring road traffic \eqref{model-ge}--\eqref{eq:cons}. 
Here, ``stationary" and ``equilibrium" means that the headway and the velocity does not change with respect to time, \textcolor{black}{hence, for an equilibrium flow $(\bar{h}_{j},  \bar{v}_{j})_{j\in\{1,...,n\}}$,} {\color{black} we have 
\begin{equation}\label{eq:stat-desc}
    \bar{h}_j(t)= \bar{h} \textrm{ and } \bar{v}_j(t)= \bar{v},\;\;\textcolor{black}{\forall t\in[0,+\infty)}.
\end{equation} }
Before going into the details, let us introduce the ``{\it velocity preferred headway}" of $f_j$: for a given velocity $v$, the so-called  ``velocity preferred headway" is given, \textcolor{black}{when it exists}, by $h$ such that
\begin{equation}
    f_j(h, 0, v)=0,
\end{equation}
which, according to \eqref{nat-con-f}, admits \textcolor{black}{at most} a unique value, and we denote it by $g_j(v)$ \textcolor{black}{when it exists}.\\

Since the headway does not change, we have 
\begin{equation}
    \dot{h}_j(t)= \dot{x}_{j+1}(t)- \dot x_j(t)= v_{j+1}(t)-  v_j(t),
\end{equation}
which, to be combined with \eqref{eq:stat-desc}, imply that \textcolor{black}{there exists a constant $\bar{v}$ such that},
\begin{gather}
    v_j(t)= \bar{v}, \; \forall j\in \{1, 2,..., n\}, \; \forall t\in \textcolor{black}{\mathbb{R}_{+}} \\
    h_j(t)= \bar h_j, \; \forall t\in \textcolor{black}{\mathbb{R}_{+}}.
\end{gather}
Furthermore, thanks to \eqref{model-ge}, we have 
\begin{equation}\label{e:111}
    f_j(\bar h_j, 0, \bar v)=0,
\end{equation}
\textcolor{black}{thus $\bar{v}$ needs to be chosen in such a way that  $g^{j}(\bar{v})$ exists for all $j\in\{1,..n\}$ and that}
\begin{equation}
\bar h_j= g_j(\bar v).
\end{equation}
Therefore, the equilibrium flow of the traffic is given by,
\begin{equation}\label{eq:sta-mol}
    (h_j(t), v_j(t))= (g_j(\bar v), \bar v)
\end{equation}
on a ring road having length 
\begin{equation}\label{eq:Lsum}
    L= \sum_{j=1}^n g_j(\bar v).
\end{equation}

{\color{black}Keeping in mind the preceding characterization of equilibrium flow, we are able to address the stability of the ring road traffic around such flows.
\begin{defn}[Local stability around the equilibrium flow] \label{def:pvstab}
Let us consider an equilibrium flow, $\{(\bar x_j(t), \bar v)\}_{j=1}^n$. The ring road traffic  is said to be exponentially stable around this equilibrium flow, if for any initial state, $(x_1(0), ..., x_n(0),v_1(0),..., v_n(0))^{T}$ being sufficiently close to $(\bar x_1(0), ..., \bar x_n(0), \bar v,..., \bar v)^{T}$, the traffic satisfies
\begin{align}
    &\;\;\;\; |(x_1(t)-\bar x_1(t)-c,..., x_n(t)- \bar x_n(t)- c, v_1(t)- \bar v,..., v_n(t)- \bar v)|\\
    &\leq C e^{-\lambda t}  |(x_1(0)-\bar x_1(0),..., x_n(0)- \bar x_n(0), v_1(0)- \bar v,..., v_n(0)- \bar v)|
\end{align}
with some constant $c\in \mathbb{R}$ depending on the initial state. The constant $c$  comes from the fact that an equilibrium flow defines a headway and a velocity, but the location of the cars is only defined up to a constant. 
\end{defn}
\textcolor{black}{We remark here that this definition is equivalent to Definition \ref{def:stanon} given below of the stability of the traffic in terms of headway-velocity.}
}

\begin{rmk}[Ring road condition]\label{rmk-rrc}
We remark that for fixed $\{f_j\}_{j=1}^n$ and $L$ there is at most one equilibrium flow, $i.e.$ at most one value of $\bar v$ such that \eqref{eq:Lsum} holds. \textcolor{black}{This means that the length of the road imposes the equilibrium flow and the steady velocity $\bar{v}$ and reciprocally that imposing a given speed $\bar{v}$ and a given system of $n$ cars determines the length of the road.}
\textcolor{black}{S}ince in reality we may \textcolor{black}{want to the study the limit where} both $n$ and $L$ go to $\infty$, what really matters \textcolor{black}{for us} is the desired steady velocity. 
\textcolor{black}{For this reason} we do not fix the value of $L$ in this paper but we fix instead the value of $\bar{v}$, which in turns, defines the value of $L$ 
via \eqref{eq:Lsum}. \textcolor{black}{This situation} will be named as ``ring road condition" \textcolor{black}{in the following}.\\
Another important property of ``ring road condition" is that the value of $L$, \eqref{eq:Lsum}, does not depend on the order of $\{f_j\}_{j=1}^n$, namely if $\{\tilde f_j\}_{j=1}^n= \{f_j\}_{j=1}^n$ \textcolor{black}{up to some permutations} then their lengths of  ring road coincident. 
\end{rmk}

\begin{rmk}[\textcolor{black}{Number of parameters describing the equilibrium flow}]
Looking that \eqref{e:111}, the stationary headway $\bar h_j= g_j(\bar v)$ \textcolor{black}{only depends on the driving characterization function $f_j$. This means that it is identical for all vehicles with the same class of parameters.}
In particular
\begin{itemize}
    \item For {\it unified model}, $g_j(\bar v)$ does not depend on $j\in \{1, 2,..., n\}$. \textcolor{black}{Thus, only 2 parameters describe the n-vehicles equilibrium flow: $\bar v$ and $g(\bar{v})$.}
    \item For general {\it collaborative driving} 
    as indicated in \eqref{ex:cd} or \eqref{ex:cdB}--\eqref{ex:cdB2}, $g^j(\bar v)$ has $m$ different choices. \textcolor{black}{Hence, the n-vehicles equilibrium flow is described by m+1 parameters.}
    \item For the unified car Bando-FTL model of collaborative driving described by \eqref{ex:cdB}--\eqref{ex:cdB2} with $V_j= V$, \textcolor{black}{where the drivers have different driving habits but the same velocity preference}, $g_j(\bar v)$ does not depend on $j\in \{1, 2,..., n\}$ \textcolor{black}{and the equilibrium flow is still only described by two parameters, despite different driving habits.}
    Actually, \textcolor{black}{in this particular case,} we observe from \eqref{ex:cdB} that $g_j(\bar v)= V^{-1}(\bar v)$. As a direct consequence, the ``ring road condition" \eqref{eq:Lsum} simply becomes $L= n  V^{-1}(\bar v)$.
    
\end{itemize}
\end{rmk}

\subsection{Traffic  around equilibrium flow}\label{sec-around}
\textcolor{black}{In this section we describe the linearized system around the equilibrium flows presented in the previous section}. Let a general traffic model \textcolor{black}{be given} by \eqref{model-ge}--\eqref{eq:cons}. \textcolor{black}{Let $(\bar{h}_{i},\bar v)$ an equilibrium flow. From the previous Section, $(\bar{h}_{i},\bar v)$ satisfies \eqref{eq:sta-mol}--\eqref{eq:Lsum}.} 
In particular, we remark that $(\bar{h}_{i},\bar v)$ is entirely parametrized by $\bar{v}$.\\

{\color{black}

By denoting the perturbation \textcolor{black}{in headway and velocity} as
\begin{equation}\label{eq:diffesystem}
    y_j(t)= h_j(t)- \bar h_j\textrm{ and } u_j(t)= v_j(t)- \bar v,
\end{equation}
the traffic flow characterized by
\eqref{model-ge}--\textcolor{black}{\eqref{eq:cons} can be written} in terms of $(y_j, u_j)$:
\begin{equation}\label{eq:gene-nonlinear}
\begin{cases}
 \dot y_j= u_{j+1}- u_j, \\
 \dot u_j= f_j(\bar h_j+ y_j, u_{j+1}- u_j, \bar v+ u_j),
\end{cases}
\end{equation}
and \eqref{eq:cons} becomes
\begin{equation}
\label{eq:new-cons}
    \sum_{j=1}^n y_j(t)= \textcolor{black}{\sum_{j=1}^n h_{j}(t)}- \sum_{j=1}^n \bar h_j=0.
\end{equation}
\textcolor{black}{This condition is equivalent to say} that the solutions of the preceding system always stay in the $(2n-1)$ dimensional subspace of $\mathbb{C}^{2}$:
\begin{equation}
\label{defH}
    \mathcal{H}:= \left\{(y_1, y_2,...,y_n,u_1,u_2,...,u_n)\in \mathbb{C}^{2n}: \sum_{k=1}^n y_k=0\right\},
\end{equation}
in particular, when the initial state takes value from $\mathcal{H}\cap \mathbb{R}^{2n}$, the solution also stays in  $\mathcal{H}\cap \mathbb{R}^{2n}$.  {\color{black}Here, we introduce complex valued spaces for the ease of notations when considering eigenvectors and eigenvalues. }

\textcolor{black}{Therefore the system \eqref{eq:gene-nonlinear}--\eqref{eq:new-cons} can be expressed as
\begin{equation}
\label{eq:new-nonlinear}
\begin{cases}
 \dot y_j= u_{j+1}- u_j, \\
 \dot u_j= f_j(\bar h_j+ y_j, u_{j+1}- u_j, \bar v+ u_j),\\
 (y_{i}(t),u_{j}(t))\in\mathcal{H}.
\end{cases}
\end{equation}}

{\bf Linearized traffic around equilibrium flow}

Next, standard linearization yields the linearized ring road traffic system 
\begin{equation}\label{eq:gene-linear}
\begin{cases}
 \dot y_j= u_{j+1}- u_j, \\
 \dot u_j= \alpha_j y_j- \beta_j u_j+ \gamma_j u_{j+1},\\
  \textcolor{black}{(y_{i}(t),u_{j}(t))\in\mathcal{H},}
\end{cases}
\end{equation} }
where $(\alpha_j, \beta_j, \gamma_j)$ (independent of time) are given by
\begin{equation}
 (\alpha_j, \beta_j, \gamma_j)= \left(\frac{\partial f_j}{\partial h}, \frac{\partial f_j}{\partial \dot{h}}- \frac{\partial f_j}{\partial v}, \frac{\partial f_j}{\partial \dot{h}}\right)(g_j(\bar v), 0, \bar v)
\end{equation}
 which satisfies, from \eqref{nat-con-f},
\begin{equation}\label{eq:com-sense}
\alpha_j>0, \;   \beta_j>\gamma_j>0.
\end{equation}

Form now on, for the ease of notations, we will denote by \textcolor{black}{$\Lambda_{i}$ the set of parameters describing the car following model $f_{i}$ and $\Delta_{j}$ the quantity describing its stability around the equilibrium flow.}  {\color{black}
\begin{defn}\label{def-delta}
For any given trio 
\begin{equation}\label{defLambda}
    \Lambda_j:= (\alpha_j, \beta_j, \gamma_j)
\end{equation}
we define its discriminant as
\begin{equation}
    \label{defdet}
    \Delta_j:= (\beta_{j})^2-(\gamma_{j})^2-2\alpha_{j}.
\end{equation}
\end{defn} }

\textcolor{black}{The expression of $\Delta_{j}$ comes from \cite{CuiSeibold2017} (see Proposition \ref{lem1}), it describes the stability of $f_{j}$ around an equilibrium flow in a single-phase traffic, i.e. when all cars have the same car-following model $f := f_{j}$. What we will see in the following is that $\Delta_{j}$ is also a good indicator of the stability in a multi-phase traffic. For this reason, we introduce the} following definition\textcolor{black}{,} inspired by Proposition \ref{lem1},
\begin{defn}\label{def-discri-delta}
  We  classify a trio, $\textcolor{black}{\Lambda :=} (\alpha, \beta, \gamma)\in \mathbb{R}^3$ satisfying the ``common sense" condition \eqref{eq-con-com-se}, {\color{black} by \textcolor{black}{the value of its} discriminant \textcolor{black}{as follows. We say that $\Lambda$ is} }
\begin{itemize}
    \item {\bf stable}, if $\textcolor{black}{\Delta  :=} \beta^2-\gamma^2- 2\alpha>0$.
     \item {\bf \textcolor{black}{critical}}, if 
     $\textcolor{black}{\Delta = 0}$.
      \item {\bf unstable}, if 
      $\textcolor{black}{\Delta < 0}$.
\end{itemize}
\end{defn}
Finally, for a general traffic model \eqref{model-ge} with $n$ cars satisfying the ring road condition \eqref{rmk-rrc}, \textcolor{black}{note that there is at most $n$ different set of parameters $\Lambda_j$ (one per car)}. However, if we  restrict it into  some $m$-phase mixed traffic model \eqref{ex:cd}, \textcolor{black}{where all the vehicles can be classified in $m$ different types, then, thanks to \eqref{e:111} --\eqref{eq:sta-mol},  $\Lambda_j$} only have at most $m$ different values.\\

Let us now denote
\begin{equation}
    z(t)= (y_1,..., y_n, u_1,... u_n)^{T}(t)\in \mathcal{H}
\end{equation}
the linearized  traffic system \eqref{eq:gene-linear} becomes
\begin{equation}\label{Cauchy-zn}
    \dot{z}(t)= M_n z(t),\;\;z(t)\in\mathcal{H},
\end{equation}
with
\begin{equation} 
M_n:=    \begin{pmatrix} O_n &A_n\\
C_n&B_n
\end{pmatrix}
\end{equation}
where
\begin{equation}\label{repre-matrix2}
 O_n:=    \begin{pmatrix} 0& & & & \\& 0& & &\\& & 0& & \\& & ...&... &\\ & & & & 0
\end{pmatrix}, \;  
 A_n:=    \begin{pmatrix} -1& 1& & & \\& -1& 1& &\\& & -1& 1& \\& & ...&... &\\ 1& & & & -1
\end{pmatrix}.
\end{equation}
\begin{equation}\label{repre-matrix}
 C_n:=    \begin{pmatrix} \alpha_1& & & & \\& \alpha_2& & &\\& & \alpha_3& & \\& & ...&... &\\ & & & & \alpha_n
\end{pmatrix}, \;  
 B_n:=    \begin{pmatrix} -\beta_1& \gamma_1& & & \\& -\beta_2& \gamma_2& &\\& & -\beta_3& \gamma_3& \\& & ...&... &\\ \gamma_n& & & & -\beta_n
\end{pmatrix}.
\end{equation}

The following definitions describe stability properties of both the linearized and nonlinear systems around equilibrium flows.  
\begin{defn}[Linearized stability around the equilibrium flow]\label{def:sta}
Let \textcolor{black}{us consider} 
an equilibrium flow on the ring road:  $\{(\bar h_j, \bar v)\}_{j=1}^n$.
The linearized ring road traffic \eqref{eq:gene-linear} is said to be exponentially stable around this equilibrium flow,  if there exist some $\lambda>0$ and $C>0$ such that  for any initial state, $z(0)= (y_1(0), ..., y_n(0),u_1(0),..., u_n(0))^{T}\in \mathcal{H}$, the solution of the Cauchy problem \eqref{Cauchy-zn} (by recalling \eqref{eq:gene-linear}--\eqref{repre-matrix}) satisfies
\begin{equation}
    |z(t)|\leq C e^{-\lambda t} |z(0)|, \forall t\in \mathbb{R}^{+}.
\end{equation}
\end{defn}
The preceding definition describes the stability of the linearized system \eqref{eq:gene-linear}. Actually, thanks to the standard linearization argument, when the linearized system is stable in the sense of Definition \ref{def:sta}, the original nonlinear system  \eqref{eq:gene-nonlinear} is automatically {\color{black}locally} stable in $\mathcal{H}$ in the following sense:
\begin{defn}[{\color{black}Local stability around the equilibrium flow: an alternative definition}]\label{def:stanon}
Let \textcolor{black}{us consider} an equilibrium flow of the ring road:  $\{(\bar h_j, \bar v)\}_{j=1}^n$.
The  ring road traffic \eqref{eq:new-nonlinear} is said to be locally exponentially stable around this equilibrium flow,  if there exist some $\varepsilon>0$, $\lambda>0$ and $C>0$ such that  for any initial state, $z(0)= (y_1(0), ..., y_n(0),u_1(0),..., u_n(0))^{T}\in \mathcal{H}$, satisfying
\begin{equation}
    |z(0)|\leq \varepsilon,
\end{equation}
the solution of the Cauchy problem \eqref{eq:gene-nonlinear} satisfies
\begin{equation}
    |z(t)|\leq C e^{-\lambda t} |z(0)|, \forall t\in \mathbb{R}^{+}.
\end{equation}
\end{defn}

Let us remark that  in the preceding two definitions, it is equivalent to express the stability in terms of position-velocity around the equilibrium flow, $\{(\bar x_j(t), \bar v)\}_{j=1}^n$, like Definition \ref{def:pvstab}.

\subsection{On the characterization of eigenvalues (counting multiplicity)}\label{sec-studyeigen}
Since the local exponential stability of the nonlinear system \eqref{eq:new-nonlinear} can be directly deduced from the exponential stability of the linearized system \eqref{Cauchy-zn} we are going to focus on the latter. Using Routh--Hurwitz criterion the exponential stability of the linearized system \eqref{Cauchy-zn} depends on the spectrum of the operator
\begin{equation}
\begin{split}
   \mathcal{L}\;:\;  &\mathcal{H}\rightarrow \mathcal{H}\\
    &x \mapsto M_{n}x.
    \end{split}
\end{equation}
Therefore we  investigate the spectrum of the matrix $M_{n}$ on $\mathcal{H}$.\\

In this section we provide {\color{black}a two-step approach} to find \textcolor{black}{the eigenvalues of the matrix $M_{n}$}. \textcolor{black}{In particular, t}he second approach gives a full description of the multiplicity of the eigenvalues. We first look at the spectrum of $M_{n}$ on $\mathbb{C}^{2n}$ and then see which eigenvalue remains when we restrict $M_{n}$ to $\mathcal{H}$.\\

{\bf The spectrum of $M_n$ \textcolor{black}{on $\mathbb{C}^{2n}$}.}

\textcolor{black}{A first method consists in considering the} eigen-pairs $(\omega, z)$ of the matrix $M_n$:
\begin{equation}
    M_n z= \omega z.
\end{equation}
{\color{black}
We know from the preceding equation that, for any $j\in \{1, 2,..., n\}$,
\begin{gather}
    \omega y_j= u_{j+1}- u_j, \\
    \omega u_j= \alpha_j y_{j}-\beta_j u_j+ \gamma_j u_{j+1},
\end{gather}
which further implies
\begin{equation}
    (\omega^2 +\beta_j\omega+ \alpha_j) u_j= (\gamma_j \omega+ \alpha_j) u_{j+1}.
\end{equation} }
Therefore 
\begin{equation}\label{eq:showeing}
    \prod_{j=1}^{n} F_j(\omega)=1, \;\; F_j(\omega):= \frac{\gamma_j \omega+ \alpha_j}{\omega^2 +\beta_j\omega+ \alpha_j},
\end{equation}
which algebraically provides $2n$ solutions (with multiplicity). However, this \textcolor{black}{approach} does not yet give any information on the multiplicity of the eigenvalues. Thus, we use a second method to obtain this information.\\

\textcolor{black}{We now use the} characteristic polynomial to find all eigenvalues (counting multiplicity) of the matrix $M_n$. We are going to \textcolor{black}{compute} {\color{black}
\begin{equation}
    \textcolor{black}{\chi(\lambda):=} \text{det} \begin{pmatrix} \lambda I_n&- A_n\\
-C_n& \lambda I_n- B_n
\end{pmatrix}
\end{equation}
Notice that $\lambda I_n C_n= C_n \lambda I_n$, we know that 
\begin{equation}
    \text{det} \begin{pmatrix} \lambda I_n&-A_n\\
-C_n& \lambda I_n- B_n
\end{pmatrix}
= \text{det} \left( (\lambda I_n)(\lambda I_n- B_n) - (- C_n)(- A_n)\right)= \text{det} (\lambda^2 I_n- \lambda B_n- C_n A_n)= \text{det}(\widetilde{M}_n),
\end{equation} }
where,  $\widetilde{M}_n$ is given by 
\begin{equation}
    \begin{pmatrix} \lambda^2 + \beta_1 \lambda+ \alpha_1& -\gamma_1 \lambda-\alpha_1& & & \\& \lambda^2 + \beta_2 \lambda+ \alpha_2& -\gamma_2 \lambda -\alpha_2& &\\& & \lambda^2 + \beta_3 \lambda+ \alpha_3& -\gamma_3 \lambda -\alpha_3& \\& &... &... &\\ -\gamma_n \lambda -\alpha_n& & & & \lambda^2 + \beta_n \lambda+ \alpha_n
\end{pmatrix}. 
\end{equation}
By considering the first column of the matrix, its determinant read as
\begin{align}
    \textcolor{black}{\chi(\lambda) =} det \widetilde{M}_n&= \prod_{i=1}^n (\lambda^2 + \beta_i \lambda+ \alpha_i)+ (-1)^{n-1} (-\gamma_n \lambda -\alpha_n) \prod_{j=1}^{n-1} (-\gamma_j \lambda -\alpha_j)  \notag\\
    &= \prod_{i=1}^n (\lambda^2 + \beta_i \lambda+ \alpha_i)- \prod_{j=1}^{n} (\gamma_j \lambda +\alpha_j). \label{eq:chara}
\end{align}
Therefore, all the eigenvalues are given by  
\begin{equation}\label{eq:eigenrepre}
      \prod_{j=1}^{n} \frac{\gamma_j \lambda+ \alpha_j}{\lambda^2 +\beta_j\lambda+ \alpha_j}\textcolor{black}{=1},
\end{equation}
which is coincident with \eqref{eq:showeing}.

\begin{rmk}[Independence of the stability with the order of the cars]
Similarly to \textcolor{black}{the length of the road for a given steady-state (see Remark \ref{rmk-rrc})}, again, we remark that the spectrum information of the traffic does not depend on the order of $\{f_j\}_{j=1}^n$. Namely, for a given desired steady velocity $\bar v$, if $\{\tilde f_j\}_{j=1}^n= \{f_j\}_{j=1}^n$ up to some permutations, then their stability around the related equilibrium flows coincide. \textcolor{black}{This means that, in the mixed-traffic setting, the stability only depends} on the penetration rate of  \textcolor{black}{the different types} of cars.
\end{rmk}

{\color{black}
{\bf \textcolor{black}{The spectrum of $M_{n}$ on $\mathcal{H}$}}

By looking at Equation \eqref{eq:eigenrepre}, it is easy to observe that $\lambda=0$ is an eigenvalue of the matrix $M_n$.
In the following, we prove that $\lambda=0$ is  a simple eigenvalue of the matrix $M_n$ acting on $\mathbb{C}^{2n}$ 
however, it is not a eigenvalue of $M_n$ acting on $\mathcal{H}$.}  \textcolor{black}{This is an important point as, otherwise, we would not be able to deduce the stability of the system \eqref{Cauchy-zn} from the eigenvalue analysis.}

\textcolor{black}{Let us start by showing that $\lambda=0$} is a simple eigenvalue of the matrix $M_n$. By comparing  the coefficients of the  characteristic polynomial \textcolor{black}{$\chi(\lambda)$ given by \eqref{eq:chara}, we immediately notice that 
$\chi(0)=0$. Then it suffices to show that $\chi'(0)\neq 0$.
Indeed, suppose that $\lambda=0$ has (at least) multiplicity two, then the characteristic polynomial can be written as $\chi(\lambda) = \lambda^{2}P(\lambda)$ where $P$ is again a polynomial, and consequently $\chi'(0) = 0$. Proving that $\chi'(0)\neq 0$} is equivalent to prove that 
\begin{equation}
    \sum_{i=1}^n \beta_i\left(\frac{\prod_{k=1}^n \alpha_k}{\alpha_i}\right)\neq  \sum_{i=1}^n \gamma_i\left(\frac{\prod_{k=1}^n \alpha_k}{\alpha_i}\right),
\end{equation}
which is guaranteed by 
the ``common sense" condition \eqref{eq:com-sense}.\\

{\color{black}
Next, we show that even though 0 is a simple eigenvalue of the matrix $M_n$, it is not included in the finite spectrum of the operator $M_n$ acting on $\mathcal{H}$. Indeed, suppose \textcolor{black}{by contradiction} that there exists $z= (y_1,..., y_n, u_1,..., u_n)^{T}\in \mathcal{H}$ \textcolor{black}{such that $z$ is an eigenvector of $M_{n}$ associated to the eigenvalue $0$. W}e have 
\begin{equation}
    M_n z= 0.
\end{equation}
\textcolor{black}{As $z\neq 0$ and using \eqref{repre-matrix}--\eqref{repre-matrix2} we deduce that} there exists some $C\neq 0$ such that for any $j\in \{1,2,...,n\}$,
\begin{equation}
    u_j= C, \; y_j= \frac{\beta_j- \gamma_j}{\alpha_j} C.
\end{equation}
Without loss of generality, we assume that $C>0$. By recalling the ``common sense" condition \eqref{eq:com-sense}, this implies that
\begin{equation}
    \sum\limits_{j=1}^{n}y_{j} > 0,
\end{equation}
but as $z\in\mathcal{H}$ we know from \eqref{defH} that
\begin{equation}
    \sum\limits_{j=1}^{n}y_{j} = 0,
\end{equation}
which leads to a contradiction.
} \textcolor{black}{This implies that
\begin{equation}
    \text{Sp}_{\mathcal{H}}(M_{n}) \subseteq \text{Sp}_{\mathbb{C}^{2n}}(M_{n})\setminus\{0\}
    ,
\end{equation}
where $\text{Sp}_{\mathcal{H}}(M_{n})$ is the spectrum of $M_{n}$ on $\mathcal{H}$ and $\text{Sp}_{\mathbb{C}^{2n}}(M_{n})$ is the spectrum of $M_{n}$ on $\mathbb{C}^{2n}$.
On the other hand, for any $\lambda\in \text{Sp}_{\mathbb{C}^{2n}}(M_{n})\setminus\{0\}$ we can find at least one related  eigenvector $z\in \mathbb{C}^{2n}$. It is clear that $M_n z\in \mathcal{H}$, thus $z\in \mathcal{H}$. Therefore,
\begin{equation}\label{spectrum}
    \text{Sp}_{\mathcal{H}}(M_{n}) = \text{Sp}_{\mathbb{C}^{2n}}(M_{n})\setminus\{0\}.
\end{equation}
}

\section{Two-phase traffic flow}\label{sec2pase}
In this section, we study the stability of the equilibrium flows in a two-phase traffic flow. This situation represents for instance two class of vehicles such as trucks and cars, or also the coexistence of vehicles with and without a collaborative driving behavior.  In the following,  these two classes of vehicles will be called {\it Type 1 vehicle} and {\it Type 2 vehicle}. Let $\bar{v}$ be an equilibrium velocity, from \eqref{eq:sta-mol} this imposes the equilibrium headway $\bar h_1$ ($resp.$ $\bar h_2$) of the {\it Type 1 vehicles} ($resp.$ {\it Type 2 vehicles}). Thus, we are looking at a situation where, for every $j\in \{1, 2,..., n\}$, using the notation \eqref{defLambda}--\eqref{defdet}, 
\begin{equation}
    \Lambda_j= (\alpha_j, \beta_j, \gamma_j) \in \{\Lambda^1, \Lambda^2\}= \{(\alpha^1, \beta^1, \gamma^1), (\alpha^2, \beta^2, \gamma^2)\}.
\end{equation}

Let us denote by $n_{1}$ the number of {\it Type 1 vehicles}  with parameters $\Lambda^{1}$ and $n_{2}$ the number of {\it Type 2 vehicles} with parameters $\Lambda^{2}$ such that the total number of vehicles is $n=n_1+n_2$.\\

Suppose that  the mixed traffic on   road is represented by the ``ordering" $ (a_1, a_2,..., a_n)$ that belongs to
\begin{equation}\label{ordering}
   \mathcal{K}:= \left\{ (a_1, a_2,..., a_n): a_k\in \{1, 2\} \, \forall 1\leq k\leq n,  \sum_{k=1}^n a_k= n_1+ 2n_2 \right\},
\end{equation}
where $a_k\in \{1, 2\}$ implies that  the $k$-th vehicle on the road is  of {\it Type $a_k$}: because there is no lane changing in a  single ring road, $(a_1, a_2,..., a_n)$ is invariant with respect to time.  Consequently, there is a unique equilibrium flow corresponding to $\bar{v}$ \footnote{or equivalently there is a unique equilibrium flow corresponding to $L$, from Remark \ref{rmk-rrc} given the ring-road condition}: the headway before the $k$-th vehicle is given by $\bar h_{a_k}$. Furthermore, from Section  \ref{sec-around}, the linearized system around this equilibrium flow is 
\begin{gather}
\begin{cases}
 \dot y_j= u_{j+1}- u_j, \\
 \dot u_j= \alpha_j y_j- \beta_j u_j+ \gamma_j u_{j+1},\\
  (\alpha_j, \beta_j, \gamma_j)=  (\alpha^{a_j}, \beta^{a_j}, \gamma^{a_j}), \label{eq:54}\\
  (y_{j}(t),u_{j}(t))_{j\in\{1,...,n\}}\in \mathcal{H},
\end{cases}
\end{gather}
which can be further represented in forms of \eqref{Cauchy-zn}--\eqref{repre-matrix}.
\\

We introduce the following Lemma
\begin{lem}\label{lem:cre}
Let given $(n_1, n_2)\in \mathbb{N}^2$.
 If the following inequality holds,
\begin{equation}\label{eq:keyor}
    \left(\frac{(\alpha^1)^2+ (\gamma^1)^2 x^2}{(\alpha^1)^2+ ((\beta^1)^2- 2\alpha^1) x^2+ x^4}\right)^{n_1} \left(\frac{(\alpha^2)^2+ (\gamma^2)^2 x^2}{(\alpha^2)^2+ ((\beta^2)^2- 2\alpha^2) x^2+ x^4}\right)^{n_2} < 1, \forall x\in \mathbb{R}\setminus\{0\},
\end{equation}
then  System \eqref{ordering}--\eqref{eq:54} is exponentially stable in the sense of Definition \ref{def:sta}.

On the other hand, if for some $x\in \mathbb{R}$ we have
\begin{equation}\label{eq:keyorine}
    \left(\frac{(\alpha^1)^2+ (\gamma^1)^2 x^2}{(\alpha^1)^2+ ((\beta^1)^2- 2\alpha^1) x^2+ x^4}\right)^{n_1} \left(\frac{(\alpha^2)^2+ (\gamma^2)^2 x^2}{(\alpha^2)^2+ ((\beta^2)^2- 2\alpha^2) x^2+ x^4}\right)^{n_2} > 1,
\end{equation}
then there exists some  $M_0\in \mathbb{N}$ effectively computable such that for any integer $M\geq M_0$, the System \eqref{ordering}--\eqref{eq:54} with $(n_1, n_2)$ being replaced by $(M n_1, Mn_2)$ is unstable. 
\end{lem}
\begin{proof}
This proof is essentially the same as the one given by  \cite[Section II]{CuiSeibold2017} in a {\it unified models} framework (namely $n_2= 0$). For readers' convenience we sketch its proof as follows.\\
By representing System \eqref{ordering}--\eqref{eq:54} in  form of \eqref{Cauchy-zn}--\eqref{repre-matrix}, thank to Section \ref{sec-studyeigen}, the eigenvalues (counting multiplicity) are explicitely characterized  by \eqref{eq:eigenrepre}:
\begin{equation}
    \left(\frac{\gamma^1 \lambda+ \alpha^1}{\lambda^2 +\beta^1\lambda+ \alpha^1}\right)^{n_1}   \left(\frac{\gamma^2 \lambda+ \alpha^2}{\lambda^2 +\beta^2\lambda+ \alpha^2}\right)^{n_2}=1.
\end{equation}
Inspired by \cite[Section II]{CuiSeibold2017}, we consider the following meromorphic function
\begin{equation}
  G(z)=  \left(\frac{\gamma^1 z+ \alpha^1}{z^2 +\beta^1 z+ \alpha^1}\right)^{n_1}   \left(\frac{\gamma^2 z+ \alpha^2}{z^2 +\beta^2 z+ \alpha^2}\right)^{n_2}, \;  \forall z\in \mathbb{C}.
\end{equation}
Since all the poles are located on the left half plane, $G(z)$ is holomorphic on the right half plane $\mathbb{C}^{+}= \{z\in \mathbb{C}: \Re(z)\geq 0\}$. We notice that $|G(z)|$ tends to 0 as $|z|$ tends to $+\infty$. Then, thanks to the maximum principle of holomorphic functions, the maximum of $|G(z)|$ over $\mathbb{C}^{+}$ must takes place at the imaginary axis. 
By considering $z= ix$ we get 
\begin{equation}
    |G(z)|^2=   \left(\frac{(\alpha^1)^2+ (\gamma^1)^2 x^2}{(\alpha^1)^2+ ((\beta^1)^2- 2\alpha^1) x^2+ x^4}\right)^{n_1} \left(\frac{(\alpha^2)^2+ (\gamma^2)^2 x^2}{(\alpha^2)^2+ ((\beta^2)^2- 2\alpha^2) x^2+ x^4}\right)^{n_2},
\end{equation}
which explains the inequality \eqref{eq:keyor}.

If Condition \eqref{eq:keyor} is satisfied, then we know that $|G(z)|\leq 1$ in $\mathbb{C}^{+}$ with $|G(z)|$ equals to 1 only at $z=0$. Therefore, all the eigenvalues are located in $\{z\in \mathbb{C}: \Re(z)<0\}\cup \{0\}$, which, to be combined with \eqref{spectrum}, yields the required exponential stability. \\

On the other hand, if  Condition \eqref{eq:keyorine}  is satisfied. As $n_{1}+n_{2}\neq0$, without loss of generality we assume that $n_{2}\neq 0$ and we define $\mu = n_{1}/n_{2}$. Condition \eqref{eq:keyorine} implies
\begin{equation}\label{eq:keyorine-2}
    \left(\frac{(\alpha^1)^2+ (\gamma^1)^2 x^2}{(\alpha^1)^2+ ((\beta^1)^2- 2\alpha^1) x^2+ x^4}\right)^{\mu} \left(\frac{(\alpha^2)^2+ (\gamma^2)^2 x^2}{(\alpha^2)^2+ ((\beta^2)^2- 2\alpha^2) x^2+ x^4}\right) > 1.
\end{equation}
We denote 
\begin{equation}
G_{\mu}(z)= \left(\frac{\gamma^1 z+ \alpha^1}{z^2 +\beta^1 z+ \alpha^1}\right)^{\mu}   \left(\frac{\gamma^2 z+ \alpha^2}{z^2 +\beta^2 z+ \alpha^2}\right), \;  \forall z\in \mathbb{C}.
\end{equation}

we consider the curve $\mathcal{C}\subset \mathbb{C}$:
\begin{equation}
    \mathcal{C}:= \{z\in \mathbb{C}: |\textcolor{black}{G_{\mu}}(z)|=1\}.
\end{equation}
Let us  further define an open subset of  $\mathcal{C}$ by 
\begin{equation}
    \mathcal{C}^{+}:=\{z\in \mathcal{C}: \Re(z)>0\},
\end{equation}
which is not empty thanks to Condition \eqref{eq:keyorine-2}, the fact that $\lim_{\text{Re}(z)\rightarrow+\infty} |G_{\mu}(z)|=0$ and the continuity of $G_{\mu}$.
It is natural to consider the continuous function on $\mathcal{C}^{+}$ defined as 
\begin{equation}
\begin{split}
  G_1  \;:\; & \mathcal{C}^{+}\rightarrow \mathbb{S}^1\\
  &z \mapsto G_1(z):= \textcolor{black}{G_{\mu}}(z).
    \end{split}
\end{equation}
We can find a connected open  set $\mathcal{O}\subset \mathbb{S}^1$ such that \begin{equation}
    \mathcal{O}\subset G_1(\mathcal{C}^+).
\end{equation}

By denoting the length of $\mathcal{O}$ as $|\mathcal{O}|$, the value of $M_0$ can be chosen as
\begin{equation}
    M_0:= \left[\frac{2\pi}{|\mathcal{O}|}\right]+ 1.
\end{equation}
Indeed, for any $M\geq M_0$, we know from the definition of $M$ that 
\begin{equation}
    \frac{2\pi}{M}< |\mathcal{O}|.
\end{equation}
Therefore, the set
\begin{equation}
    \{e^{2k\pi/ M}: k=1,2,...,M\}\cap \mathcal{O}
\end{equation}
is not empty. We assume that for some $k$, the point $e^{2k\pi/M}$ belongs to $\mathcal{O}$. Thus, by the definition of $\mathcal{O}$ there exists some $z_0\in \mathcal{C}^+$ such that
\begin{equation}
    G_{\mu}(z_0)= G_1(z_0)= e^{2k\pi/M}.
\end{equation}
Meanwhile, we recall that for the ring road traffic with $(M \mu)$ {\it Type 1 vehicles} and $M$ {\it Type 2 vehicles}  the stability of the System \eqref{ordering}--\eqref{eq:54} is determined by the solutions of $G_M(z)=1$:
\begin{equation}\label{eq:keyorineend}
    G_M(z)=  \left(\frac{\gamma^1 z+ \alpha^1}{z^2 +\beta^1 z+ \alpha^1}\right)^{M \mu}   \left(\frac{\gamma^2 z+ \alpha^2}{z^2 +\beta^2 z+ \alpha^2}\right)^{M}.
\end{equation}
The preceding equations immediately yield $G_M(z_0)=1$ with $\Re(z_0)> 0$:  the system is unstable.
\end{proof}

This lemma allows us to show the following theorems:

\begin{thm} 
\label{th11}
Let given $\bar v>0$.
If $\Delta^1\geq 0$ and $\Delta^2\geq 0$,  then for any $(n_1, n_2)\in \mathbb{N}^2$ and any ordering of the vehicles on the road, the ring road traffic system \eqref{eq:new-nonlinear} is locally exponentially stable around the equilibrium flow.
\end{thm}

\begin{rmk}
Note that, in the critical case, i.e. one or both of  $\Delta^{1}$ and $\Delta^{2}$ equals to 0, the system is still exponentially stable.
\end{rmk}

\begin{thm}
\label{th2}
Let given $\bar v>0$.
 We assume that $\Delta^1\geq 0$ and $\Delta^2<0$.
\begin{itemize}
      \item[(1)] If $\Delta^{1}> 0$ then there exists some effectively computable threshold constant $\tau_0\in (0, 1)$ depending on $\Lambda^{1}$ and $\Lambda^{2}$ and given by \eqref{eq:deftau0} such that for any pair $(n_1, n_2)\in \mathbb{N}^2$ verifying
\begin{equation}
    \frac{n_1}{n_1+ n_2} >  \tau_0, 
\end{equation}
the inequality \eqref{eq:keyor} is satisfied. In other words, for any ordering of the vehicles on the road $(a_1,a_2,..., a_n)\in \mathcal{K}$, the ring road traffic system \eqref{eq:new-nonlinear} is locally exponentially stable around the equilibrium flow associated to $\bar{v}$.
 \\

On the other hand, for any penetration rate
\begin{equation}
\tau := \frac{n_{1}}{n_{1}+n_{2}}    <\tau_{0}
\end{equation}
there exists $M>0$ such that for any ${n_1, n_2}\in \mathbb{N}^2$ satisfying
$n_1+n_2>M$,
the ring road traffic  system \eqref{model-ge}--\eqref{eq:cons} is unstable around the  equilibrium flow $(\bar{h}_{i},\bar{v})_{i\in\{1,2\}}$.

       \item[(2)]
       If $\Delta^1=0$ (namely, $\Lambda^1$ is critical), then for any penetration rate $\tau = n_{1}/(n_{1}+n_{2})$ there exists $M>0$ effectively computable such that if $n>M$ the ring road traffic system \eqref{eq:new-nonlinear} is unstable around the equilibrium flow associated to $\bar{v}$.
\end{itemize}
\end{thm}

Finally, even though $\tau_{0}$ can be {\color{black} easily calculated with the help of a minimization algorithm} we show some simpler upper and lower bounds.
\begin{cor}
The critical penetration rate $\tau_{0}$ defined in Theorem \ref{th2} satisfies
\begin{equation}
\begin{split}
\tau_{0} &\geq \frac{-\Delta^2 (\alpha^1)^2}{\Delta^1 (\alpha^2)^2-\Delta^2 (\alpha^1)^2},\\
\text{ and }\;\;\; \tau_{0}&\leq  \frac{(-\Delta^2) \left((\alpha^1)^2+ (\gamma^1)^2 \Gamma^2\right)\left((\alpha^1)^2+ ((\beta^1)^2- 2\alpha^1) \Gamma^2+ (\Gamma^2)^2\right)}{(\beta^2)^2 (\alpha^1)^2 \Delta^1+ (-\Delta^2) \left((\alpha^1)^2+ (\gamma^1)^2 \Gamma^2\right)\left((\alpha^1)^2+ ((\beta^1)^2- 2\alpha^1) \Gamma^2+ (\Gamma^2)^2\right)}.
\end{split}
\end{equation}
\end{cor}

{\color{black}As we can see that Theorem \ref{th1}--\ref{th13} in Section \ref{sec:main-results} are  direct consequences of the more detailed  Theorems \ref{th11}--\ref{th2}, in the following we only give the proofs of the latter theorems. }

\begin{proof}[Proof of Theorem \ref{th2}] 
 Note that it suffices to study the exponential stability of the linearized system \eqref{eq:54} since the local exponential stability of the nonlinear system follows. At first we prove point (1) of this theorem.
Looking at Lemma \ref{lem:cre}, it is thus sufficient to investigate \eqref{eq:keyor}. Let us first get an intuition about what happens depending on the proportion of stable and unstable vehicle.
 To simplify the condition, we can set $y = x^{2}$ and $\mu=n_{1}/n_{2}$, namely $\mu/(1+\mu)$ is the proportion of stable vehicle in the traffic. The condition \eqref{eq:keyor} becomes 
\begin{equation}
         \left(\frac{(\alpha^1)^2+ (\gamma^1)^2 y}{(\alpha^1)^2+ ((\beta^1)^2- 2\alpha^1) y+ y^2}\right)^{\mu} \left(\frac{(\alpha^2)^2+ (\gamma^2)^2 y}{(\alpha^2)^2+ ((\beta^2)^2- 2\alpha^2) y+ y^{2}}\right)< 1, \forall y\in \mathbb{R}_{+}^{*}.
\end{equation}
We set 
\begin{equation}
         h_{1}(y) = \left(\frac{(\alpha^1)^2+ (\gamma^1)^2 y}{(\alpha^1)^2+ ((\beta^1)^2- 2\alpha^1) y+ y^2}\right)
\end{equation}
and we define $h_{2}$ similarly. We observe that, for $i\in \{1, 2\}$,
\begin{equation}
\label{eq:hideri}
h_{i}'(y) =\frac{-(\gamma^{i})^{2}y^{2}-2(\alpha^{i})^{2}y-(\alpha^{i})^{2}\Delta^{i}}{\left((\alpha^i)^2+ ((\beta^i)^2- 2\alpha^i) y+ y^2\right)^{2}},
\end{equation}
where, we recall that  $\Delta^{i} = (\beta^{i})^2-(\gamma^{i})^2-2\alpha^{i}$. This means that $h_{i}$ has at most two points where its derivative vanishes and these potential points are given by
\begin{equation}
    y_{\pm}=-\frac{(\alpha^{i})^{2}}{(\gamma^{i})^{2}}\left(1\mp\sqrt{1-\frac{(\gamma^{i})^{2}}{(\alpha^{i})^{2}}\Delta^{i}}\right).
\end{equation}
However, note that we are only interested in the values of $h_{i}$ on $[0,+\infty)$. This gives some insight about what happens when $\Delta^{i}$ moves from a positive value (stable region) to a negative value (possibly unstable region): when $\Delta^{i}$ is positive there is no non-negative vanishing points of $h_{i}'$, which means that $h_{i}$ start at the value $h_{i}(0)=1$ and then decrease strictly continuously until it reaches the limit $\lim_{y\rightarrow+\infty} h_{i}(y)=0$. When $\Delta_{i}$ is negative, then $y_{+}$ is the only positive vanishing point of $h'$ which means that $h_{i}$ still starts at the value $h_{i}(0)=1$ but increase strictly up to $y=y_{+}$ and becomes larger than $1$. The critical case $\Delta^{i}=0$ corresponds to the special case where $y_{+}=0$ and therefore $h$ still decreases strictly on $[0,+\infty)$. Let us now prove (1) of Theorem \ref{th2}
\\

{\bf Quantitative characterization of the optimal choice of $\tau_0$}

We define
\begin{equation}
\label{defHi}
H_{i}(y) = \log(h_{i}) = \log\left(\frac{(\alpha^{i})^2+ (\gamma^{i})^2 y}{(\alpha^{i})^2+ ((\beta^{i})^2- 2\alpha^{i}) y+ y^2}\right).
\end{equation}
then for any given $y>0$, \eqref{eq:keyor} is equivalent to
\begin{equation}\label{eq:keyor00}
    \mu H_1(y)+ H_2(y)<0,\;\;\text{  }\forall\; y>0.
\end{equation}
Since $H_1(y)<0 \; \forall y>0$, the preceding condition  is equivalent to having
\begin{equation}
\label{eq:keyor2}
      \mu> -\frac{H_2(y)}{H_1(y)},\;\;\text{  }\forall\; y>0.
\end{equation}
Therefore, it {\color{black}leads us} to introduce the quantity
\begin{equation}
    K:= \sup \left\{-\frac{H_2(y)}{H_1(y)};\;  y\in \left(0, +\infty\right)\right\},
\end{equation}
and to show that this quantity is finite. If so, it suffices to choose $\mu> K$ to conclude the exponential
stability. We will show the following:  define
\begin{gather}
    N_0= \max \left\{-\frac{H_2(y)}{H_1(y)};\;  y\in \left(0, \Gamma^2\right]\right\}<+ \infty, \label{eq:op1}\\
    \Gamma^2:= \frac{-(\alpha^2)^2+  \sqrt{(\alpha^2)^4- (\alpha^2)^2(\gamma^2)^2 \Delta^2}}{(\gamma^2)^2}\in (0, -\Delta^2), \label{def:Gamma2}\\
   \tau_0:= \frac{N_0}{N_0+1}, \label{eq:op2}
\end{gather}
then $K = N_{0}<+\infty$ and if 
\begin{equation}
    \label{condt0}
    \frac{n_{1}}{n_{1}+n_{2}} > \tau_{0},
\end{equation}
condition \eqref{eq:keyor2} is satisfied and consequently the system is exponentially stable around the considered equilibrium flow.\\

From \eqref{eq:hideri}
\begin{gather}\label{eq:deri}
    H'_1(y)= \frac{-(\gamma^1)^2 y^2- 2(\alpha^1)^2 y- (\alpha^1)^2\Delta^1}{\left((\alpha^1)^2+ (\gamma^1)^2 y\right)\left((\alpha^1)^2+ ((\beta^1)^2- 2\alpha^1) y+ y^2\right)}, \\
     H'_2(y)= \frac{-(\gamma^2)^2 y^2- 2(\alpha^2)^2 y- (\alpha^2)^2\Delta^2}{\left((\alpha^2)^2+ (\gamma^2)^2 y\right)\left((\alpha^2)^2+ ((\beta^2)^2- 2\alpha^2) y+ y^2\right)}.
\end{gather}
And using this together with \eqref{defHi}, we deduce that
\begin{gather}
    H_1(0)= H_2(0)=0, \label{eq:H1H1} \\
  H_1(y)<0,\;\;   H'_1(y)<0,\;\; \forall y\in (0, +\infty),\\
  H_2(y)<0,\;\; \forall y\in (-\Delta^2, +\infty).
\end{gather}
Concerning $H_{2}$ observe that we have the following key symmetry
\begin{equation}
      H_2\left(\frac{(\alpha^2)^2(-\Delta^2-y)}{(\alpha^2)^2+ (\gamma^2)^2 y}\right)= H_2(y),\;\; \forall y\in [0, -\Delta^2]. \label{eq:H2sys}
\end{equation}
This implies, by recalling the definition of $\Gamma^2$ in \eqref{def:Gamma2}, 

we have 
\begin{equation}\label{eq:opt}
   \sup \left\{-\frac{H_2(y)}{H_1(y)};\;  y\in \left(0, \Gamma^2\right]\right\}=  \sup \left\{-\frac{H_2(y)}{H_1(y)};\;  y\in (0, +\infty)\right\},
\end{equation}
or equivalently $N_{0} = K$. {\color{black}Considering the fact $H_{2}/H_{1}$ is a continuous function, in order to prove that $N_{0}$ is bounded it only remains to show that there exists a  finite limit as $y$ tends to $0^{+}$.} From \eqref{eq:deri},
\begin{equation}\label{eq:lower-bound}
    \lim_{y\rightarrow 0^+} -\frac{H_2(y)}{H_1(y)}=   -\frac{H'_2(0)}{H'_1(0)}= \frac{-\Delta^2 (\alpha^1)^2}{\Delta^1 (\alpha^2)^2}\in (0, +\infty),
\end{equation}
which concludes that $N_0<+\infty$. As $n_{1} = \mu n_{2}$, \eqref{condt0} is equivalent to $\mu>N_{0}$.

To show that $\tau_{0}$ is optimal, is suffices to observe that if $\mu< N_{0}$ (or equivalently $n_1/(n_1+ n_2)< \tau_0$) then by continuity there exists a subset $[y_{1},y_{2}]\subset(0,\Gamma^{2}]$ with $y_{1}\neq y_{2}$ such that
\begin{equation}
\mu <- \frac{H_{1}(y)}{H_{2}(y)},\;\;\text{ for any }y\in[y_{1},y_{2}],
\end{equation}
which implies that for any $y\in [y_{1},y_{2}]$, $y>0$ and 
\begin{equation}
\label{eq:t0opti}
\left(\frac{(\alpha^1)^2+ (\gamma^1)^2 y}{(\alpha^1)^2+ ((\beta^1)^2- 2\alpha^1) y+ y^2}\right)^{\mu} \left(\frac{(\alpha^2)^2+ (\gamma^2)^2 y}{(\alpha^2)^2+ ((\beta^2)^2- 2\alpha^2) y+ y^{2}}\right)> 1.
\end{equation}
Setting $x =\sqrt{y}$, \eqref{eq:t0opti} together with Lemma \ref{lem:cre} and \eqref{eq:keyorine} allows us to conclude that there exists $M$ large enough such that for any $n_{1}>M$ and $n_{2}>M$ {\color{black}satisfying $n_1/(n_1+n_2)=\textcolor{black}{\tau}$},
the system \eqref{eq:new-nonlinear} is unstable around the equilibrium flow $(\bar{h}_{i},\bar v)_{i\in\{1,2\}}$.

\textcolor{black}{We have now proved the existence of the critical penetration rate $\tau_{0}$ and we obtained a quantitative characterization. Note that $\tau_{0}$ can be easily computed by a minimization algorithm using \eqref{eq:op1}--\eqref{eq:op2} and provides the optimal penetration rate of stable cars to stabilize the traffic. 
In the following, for the qualitatively study convenience,  we also present some lower and upper bounds of  $\tau_0$ (or equivalently, some lower and upper bounds of $N_0$).}

{\bf A simple lower bound of $\tau_0$.}

Thanks to \eqref{eq:lower-bound}, we see that
\begin{equation}
    N_{0}\geq \frac{-\Delta^2 (\alpha^1)^2}{\Delta^1 (\alpha^2)^2},
\end{equation}
hence a lower bound of $\tau_0$ can be expressed by
\begin{equation}
    B_l:= \frac{-\Delta^2 (\alpha^1)^2}{\Delta^1 (\alpha^2)^2-\Delta^2 (\alpha^1)^2}.
\end{equation}
Recall that $-\Delta^{2}<0$, so $B_{l}\in(0,1)$.

{\bf A simple upper  bound of $\tau_0$.}

The precise value of $N_0$ is given by  \eqref{eq:op1}, but it is rather  difficult 
to determine by hand. Indeed, to do so we would need to compare the extreme points of $- H_2/ H_1$, which are given by 
\begin{equation}\label{eq:extremevalue}
  \mathcal{E}:= \{z\in (0, \Gamma^2):   H'_2(z) H_1(z)- H_2(z) H'_1(z)= 0\}.
\end{equation}
In terms of those extreme points, $N_0$ is further given by 
\begin{equation}\label{N-ex2}
    N_0= \max \left\{ -\frac{H_2(y)}{H_1(y)}: y\in \mathcal{E}\cup \{0, \Gamma^2\} \right\},
\end{equation}
where 
\begin{equation}\label{eq:def-value0}
    -\frac{H_2(0)}{H_1(0)}:=-\frac{H'_2(0)}{H'_1(0)}= \frac{-\Delta^2 (\alpha^1_1)^2}{\Delta^1 (\alpha^2_1)^2}.
\end{equation}
Actually, we can exclude $\Gamma^2$ in $\{0, \Gamma^2\}$ from the expression \eqref{N-ex2}: suppose that $\Gamma^2\in \mathcal{E}$ then $\mathcal{E}\cup \{0, \Gamma^2\}= \mathcal{E}\cup \{0\}$; otherwise, there exists some $y_0\in (\Gamma^2-\delta, \Gamma^2+ \delta)$ such that $-H_2(y_0)/H_1(y_0)$ is bigger than those of $\Gamma^2$,  then thanks to the symmetric property of $H_2$ and the monotonous property fo $H_1$, \eqref{eq:H1H1}--\eqref{eq:H2sys},  there exists some $y_1\in (\Gamma^2-\delta, \Gamma^2)$ such that $-H_2(y_1)/H_1(y_1)$ is bigger than those of $\Gamma^2$. Therefore,
\begin{equation}\label{N-ex3}
    N_0= \max \left\{ -\frac{H_2(y)}{H_1(y)}: y\in \mathcal{E}\cup \{0\} \right\}.
\end{equation}

Next, notice that for any extreme point $z\in \mathcal{E}$ we have 
\begin{equation}
    - \frac{H_2(z)}{H_1(z)}=   - \frac{H'_2(z)}{H'_1(z)},
\end{equation}
which, to be combined with \eqref{eq:def-value0}, yield
\begin{equation}\label{N-ex4}
    N_0= \max \left\{ -\frac{H'_2(y)}{H'_1(y)}: y\in \mathcal{E}\cup \{0\} \right\}.
\end{equation}
We remark here that, though it is relatively easy to get the maximum value of $-H'_2/H'_1$ in $[0, \Gamma^2]$ (as its extreme points can be calculated explicitly via polynomials), this value is not necessarily equivalent to $N_0$. More precisely, 
\begin{equation}
    N_0\leq \max \left\{ -\frac{H'_2(y)}{H'_1(y)}: y\in [0, \Gamma^2] \right\}=: \widetilde N_0.
\end{equation}
$\widetilde N_0$ can also be expressed in terms of extreme points:
\begin{gather}
    \widetilde N_0= \max \left\{ -\frac{H'_2(y)}{H'_1(y)}: y\in \widetilde{\mathcal{E}} \right\}, \\
    \widetilde{\mathcal{E}}:= \{z\in [0, \Gamma^2]: H''_2(z) H'_1(z)- H'_2(z)H''_1(z)=0\}. 
\end{gather}

In the following we present a simple upper bound for $\widetilde N_0$.  Observe that $- H'_2(y)/ H'_1(y)$ is characterized by 
{\small
\begin{equation}
    \frac{-(\gamma^2)^2 y^2- 2(\alpha^2)^2 y- (\alpha^2)^2\Delta^2}{\left((\alpha^2)^2+ (\gamma^2)^2 y\right)\left((\alpha^2)^2+ ((\beta^2)^2- 2\alpha^2) y+ y^2\right)}\cdot \frac{\left((\alpha^1)^2+ (\gamma^1)^2 y\right)\left((\alpha^1)^2+ ((\beta^1)^2- 2\alpha^1) y+ y^2\right)}{(\gamma^1)^2 y^2+ 2(\alpha^1)^2 y+ (\alpha^1)^2\Delta^1},
\end{equation} }

while for $y\in [0, \Gamma^2]$ there are
{\small
\begin{gather}
    (\gamma^1)^2 y^2+ 2(\alpha^1)^2 y+ (\alpha^1)^2\Delta^1\geq (\alpha^1)^2\Delta^1, \\
    \left((\alpha^1)^2+ (\gamma^1)^2  y\right)\left((\alpha^1)^2+ ((\beta^1)^2- 2\alpha^1) y+ y^2\right)\leq \left((\alpha^1)^2+ (\gamma^1)^2 \Gamma^2\right)\left((\alpha^1)^2+ ((\beta^1)^2- 2\alpha^1) \Gamma^2+ (\Gamma^2)^2\right),\\
    -(\gamma^2)^2 y^2- 2(\alpha^2)^2 y- (\alpha^2)^2\Delta^2\leq - (\alpha^2)^2\Delta^2, \\
    \left((\alpha^2)^2+ (\gamma^2)^2 y\right)\left((\alpha^2)^2+ ((\beta^2)^2- 2\alpha^2) y+ y^2\right)\geq (\alpha^2)^2 (\beta^2)^2.
\end{gather} }

Consequently,
\begin{equation}
    \widetilde N_0\leq \frac{(-\Delta^2) \left((\alpha^1)^2+ (\gamma^1)^2 \Gamma^2\right)\left((\alpha^1)^2+ ((\beta^1)^2- 2\alpha^1) \Gamma^2+ (\Gamma^2)^2\right)}{(\beta^2)^2 (\alpha^1)^2 \Delta^1},
\end{equation}
which further implies the following upper bound of $\tau_0$: 
\begin{equation}
    B_u:=  \frac{(-\Delta^2) \left((\alpha^1)^2+ (\gamma^1)^2 \Gamma^2\right)\left((\alpha^1)^2+ ((\beta^1)^2- 2\alpha^1) \Gamma^2+ (\Gamma^2)^2\right)}{(\beta^2)^2 (\alpha^1)^2 \Delta^1+ (-\Delta^2) \left((\alpha^1)^2+ (\gamma^1)^2 \Gamma^2\right)\left((\alpha^1)^2+ ((\beta^1)^2- 2\alpha^1) \Gamma^2+ (\Gamma^2)^2\right)}.
\end{equation}

Let us now prove point (2) of Theorem \ref{th2}. We define again $\mu = n_{1}/n_{2}$. Suppose that $\Delta^1= 0$ and $\Delta^2<0$.  Thanks to \eqref{eq:deri}, we know that 
\begin{equation}
     \mu H'_1(0)+ H'_2(0)= -\frac{\Delta^2}{(\alpha^2)^2}>0,
\end{equation}
which, to be combined with the fact that $ \mu H_1(0)+ H_2(0)=0$, imply the existence of $y>0$ such that 
\begin{equation}
     \mu H_1(y)+ H_2(0)>0.
\end{equation}
As a direct consequence, the system is not stable.\\
\end{proof}

\section{Multi-phase collaborative driving}
\label{sec:multi}

In this section, similarly to the preceding Section, we study the stability of the equilibrium flows in a multi-phase mixed traffic flow.  For any given equilibrium velocity  $\bar{v}$, using the notation $\Lambda_j$ and \eqref{defdet},  we have that 
\begin{equation}
    \Lambda_j= (\alpha_j, \beta_j, \gamma_j) \in \{\Lambda^1,..., \Lambda^m\}= \{(\alpha^1, \beta^1, \gamma^1),...,(\alpha^m, \beta^m, \gamma^m)\}.
\end{equation}

Again, for $k\in \{1,2,...,m\}$ we  denote by $n_{k}$ the number of {\it Type k vehicles}  with parameters $\Lambda^{k}$ such that the total number of vehicles is $n=\sum_{j=1}^m n_j$.

For any ordering of the vehicle on road, $i.e.$ the $j$-th vehicle is of {\it Type $a_j$}, the linearized system around the unique equilibrium flow is 
\begin{gather}
\begin{cases}
 \dot y_j= u_{j+1}- u_j, \\
 \dot u_j= \alpha_j y_j- \beta_j u_j+ \gamma_j u_{j+1},\\
  (\alpha_j, \beta_j, \gamma_j)=  (\alpha^{a_j}, \beta^{a_j}, \gamma^{a_j}), \\
  (y_{j}(t),u_{j}(t))_{j\in\{1,...,n\}}\in \mathcal{H}.
\end{cases}
\end{gather}
Similar to Theorem \ref{th11} and Theorem \ref{th2} we have the following theorem concerning the stability of the $m$-phase mixed ring road traffic.

Indeed, similar to the two-phase traffic,  we look at the ring road traffic with $m$ populations having respectively $(n_1, n_2,..., n_m)$ many vehicles. The local exponential stability of this system is still equivalent to the study of the real part of eigenvalues of the linearized system that are explicitly given by the roots of the polynomial \eqref{eq:showeing} presented in Section \ref{sec-studyeigen}:
\begin{equation}
    \prod_{j=1}^{n} F_j(\omega)=1, \;\; F_j(\omega):= \frac{\gamma_j \omega+ \alpha_j}{\omega^2 +\beta_j\omega+ \alpha_j}.
\end{equation}

Again, it is not easy to calculate the exact roots of the polynomial. Instead, we investigate the value of the polynomial on the imaginary axis. The question becomes:
\begin{equation}\label{cre}
  \textrm{whether }\;   \prod_{k=1}^m \left(\frac{(\alpha^k)^2+ (\gamma^k)^2 x^2}{(\alpha^k)^2+ ((\beta^k)^2- 2\alpha^k) x^2+ x^4}\right)^{n_k}  < 1\;\text{ holds for any }\; x\in \mathbb{R}\setminus\{0\}?
\end{equation}
Thanks to the same reasoning as in Lemma \ref{lem:cre}, we get
\begin{itemize}
    \item if   Condition \eqref{cre} is satisfied, then all the roots of the polynomial excluding 0 are distributed on the left complex region $i.e.$ $\{z\in \mathbb{C}: \Re{z}<0\}$. Hence, the linearized system is exponentially stable.
    \item if   Condition \eqref{cre} is not verified, and if further there exists some $x_0\in \mathbb{R}$ such that the value of the function in \eqref{cre} is strictly large than 1, then the ring road traffic system with these penetration rates of vehicles is not stable provided sufficiently many. cars\textcolor{black}{ This is obtained using the same reasoning as in the two population case (see \eqref{eq:keyorine-2}--\eqref{eq:keyorineend}).}
\end{itemize}
 Recalling  Definition \ref{def-delta}--\ref{def-discri-delta} concerning the classification of $\Delta^k$, we know that 
\begin{itemize}
    \item if $\Delta^k>0$, then 
   \begin{equation}\label{eq:sec6-1}
       \left(\frac{(\alpha^k)^2+ (\gamma^k)^2 x^2}{(\alpha^k)^2+ ((\beta^k)^2- 2\alpha^k) x^2+ x^4}\right)^{n_k}  < 1, \forall x\in \mathbb{R}\setminus\{0\};
   \end{equation} 
     \item  if 
     $\Delta^k = 0$, then Inequality \eqref{eq:sec6-1} is also satisfied;
      \item if 
      $\Delta^k < 0$, then 
       \begin{equation}
       \left(\frac{(\alpha^k)^2+ (\gamma^k)^2 x^2}{(\alpha^k)^2+ ((\beta^k)^2- 2\alpha^k) x^2+ x^4}\right)^{n_k}  > 1, \textrm{ for some } x\in \mathbb{R}\setminus\{0\}.
   \end{equation} 
\end{itemize} 
This observation, together with Condition \eqref{cre}, finally lead to the following Theorem \ref{thm:61}. 

\begin{thm}\label{thm:61} 
Let given $\bar v>0$. We assume without loss of generality that $\Delta^1\geq\Delta^2\geq...\geq \Delta^m$.
\begin{itemize}
      \item[(1)] If $\Delta^m\geq 0$, then  for any $(n_1, n_2, n_3,..., n_m)\in \mathbb{N}^m$ and any ordering of the vehicles on the road, the ring road traffic system \eqref{eq:new-nonlinear} is locally exponentially stable around the equilibrium flow.
      \item[(2)] If $\Delta^{1}> 0$  and $\Delta^m<0$,  then there exists some effectively computable threshold constant $\tau_1\in (0, 1)$ depending on $\{\Lambda^{1},...,\Lambda^{m}\}$  such that for any {\color{black}m-tuple} $(n_1, n_2,..., n_m)\in \mathbb{N}^2$ verifying
\begin{equation}
    \frac{n_1}{\sum_{k=1}^m n_k} >  \tau_1, 
\end{equation}
the inequality \eqref{eq:keyor} is satisfied. In other words, for any ordering of the vehicles on the road $(a_1,a_2,..., a_n)\in \mathcal{K}$, the ring road traffic system \eqref{eq:new-nonlinear} is locally exponentially stable around the equilibrium flow associated to $\bar{v}$.
 \\
 On the other hand, we further assume that for some $k\in\{1,..., m-1\}$ there is $\Delta^k>0\geq \Delta^{k+1}$. There exists some $\tau_2\in (0, 1)$ effectively computable such that for any fixed penetration rate\textcolor{black}{s}
\begin{equation}
\left(\frac{n_1}{\sum_{j=1}^m n_j},..., \frac{n_m}{\sum_{j=1}^m n_j} \right)  \; \textrm{ satisfying } \;  \frac{ \sum_{j=1}^k n_j}{\sum_{j=1}^m n_j}< \tau_2
\end{equation}
there exists $M>0$ such that for any ${n_1,..., n_m}\in \mathbb{N}^m$ satisfying
$\sum_{k=1}^m n_k>M$,
the ring road traffic  system \eqref{eq:new-nonlinear}  is unstable around the  equilibrium flow associated to $\bar v$.

       \item[(3)]
       If $\Delta^1=0$ (namely, $\Lambda^1$ is critical), then for any fixed penetration rate of the cars the ring road traffic system \eqref{eq:new-nonlinear} is unstable around the equilibrium flow associated to $\bar{v}$  provided sufficiently many cars on road.
\end{itemize}
\end{thm}

Its proof is essentially similar to the proofs of Theorem \ref{th11}-\ref{th12}: the easier cases (1) and (3) are direct consequences of the preceding observations. Concerning the mixed case (2) such that both stable and unstable vehicles coexist, heuristically speaking,  if there is more stable cars on road then more likely Condition \ref{cre} is satisfied. Otherwise with fewer stable cars on road  the unstable cars will dominate the traffic to prevent us from getting Condition \ref{cre}. \textcolor{black}{This comes from the fact that we get the following condition for the stability instead of \eqref{eq:keyor00}
\begin{equation}
    \sum\limits_{i=1}^{n}n_{i}H_{i}(y)<0,\;\;\forall y\geq 0,
\end{equation}
where $H_{i}(y)=\log\left(\frac{(\alpha^{i})^2+ (\gamma^{i})^2 y}{(\alpha^{i})^2+ ((\beta^{i})^2- 2\alpha^{i}) y+ y^2}\right)$ is defined similarly as in \eqref{defHi}.} {\color{black}While in this specific case there is
\begin{gather}
    H_i(0)=0 \textrm{ and } H'_i(0)<0, \forall i\in \{1,2,...,k\}, \\
     H_i(0)=0 \textrm{ and } H'_i(0)>0, \forall i\in \{k+1,k+2,...,m\}.
\end{gather}
}

\section{Numerical experiments}\label{sec:numerics}
In this Section we present numerical experiments to illustrate Theorems \ref{th1}--\ref{th13} with two type of populations described by some set of parameters $\Lambda_{1}$ and $\Lambda_{2}$.\\

We assume for this example  that the vehicles are described by the nonlinear Bando-FTL model
\eqref{Bando-FTL} and we consider two populations described by the parameters $(a_{1}, b_{1})$ and $(a_{2},b_{2})$. We assume that they have the same velocity preference $V$ given by
\begin{equation}
    V(h)=V_{\max}\frac{\tanh(\frac{h-l_{v}}{d_{0}}-2)+\tanh(2)}{1+\tanh(2)},
\end{equation}
which is a usual choice for Bando-FTL \cite{bando1995dynamical}. We study a case where the first population has a very stable behavior on the road with $\Delta_{1}>0$ while the second population is made of slightly more aggressive driver that have a slightly unstable behavior such that $\Delta_{2}<0$ but  $|\Delta_{2}|<|\Delta_{1}|$.  To do so, we choose the parameters $a_{1}=4$, $b_{1}= 20$, $a_{2}=0.5$, $b_{2}=20$, thus the instability of the second population will simply come from a higher sensitivity to the velocity preference rather than the ``Follow-the-leader" behavior. We choose $L$ and $N$ such that $L/N=10.4 m$, which is a steady-state value similar to the setting of the real life experiment described in \cite{stern2018dissipation}. Note that parameters $(a_{2},b_{2})$ are typical values that were obtained in \cite{PohlmannSeibold} after calibration on data from real life experiments. From Theorem \ref{th2} we deduce the following:
\begin{equation}
\begin{split}
    \Delta_{1} = 7.28,\;\;\;
    \Delta_{2} = -0.84,\;\;\;\tau_{0}=0.881
    \end{split}
\end{equation}
as we can see, although $|\Delta_{1}|$ is one order of magnitude above $|\Delta_{2}|$, the penetration rate of stable vehicle needed for having a stable flow is very high and above $88\%$. This means that even a very low proportion of slightly more aggressive drivers in a large road can completely destabilize a traffic that would be very stable otherwise. In Figure \ref{fig1} (left) we show the speed variance across drivers of a traffic flow with 500 cars, a penetration rate of stable cars of $80\%$ ($\tau=0.802$) and we see that the system is quickly unstable as the speed variance only increase during the entire simulation. In Figure \ref{fig1} (right) we show the speed variance of the same traffic when the penetration rate of stable car is $88,2\%$ instead and we see that the speed variance, already low at initial time, decreases exponentially fast.
\begin{figure}[ht!]
    \centering
    \includegraphics[width=0.48\textwidth]{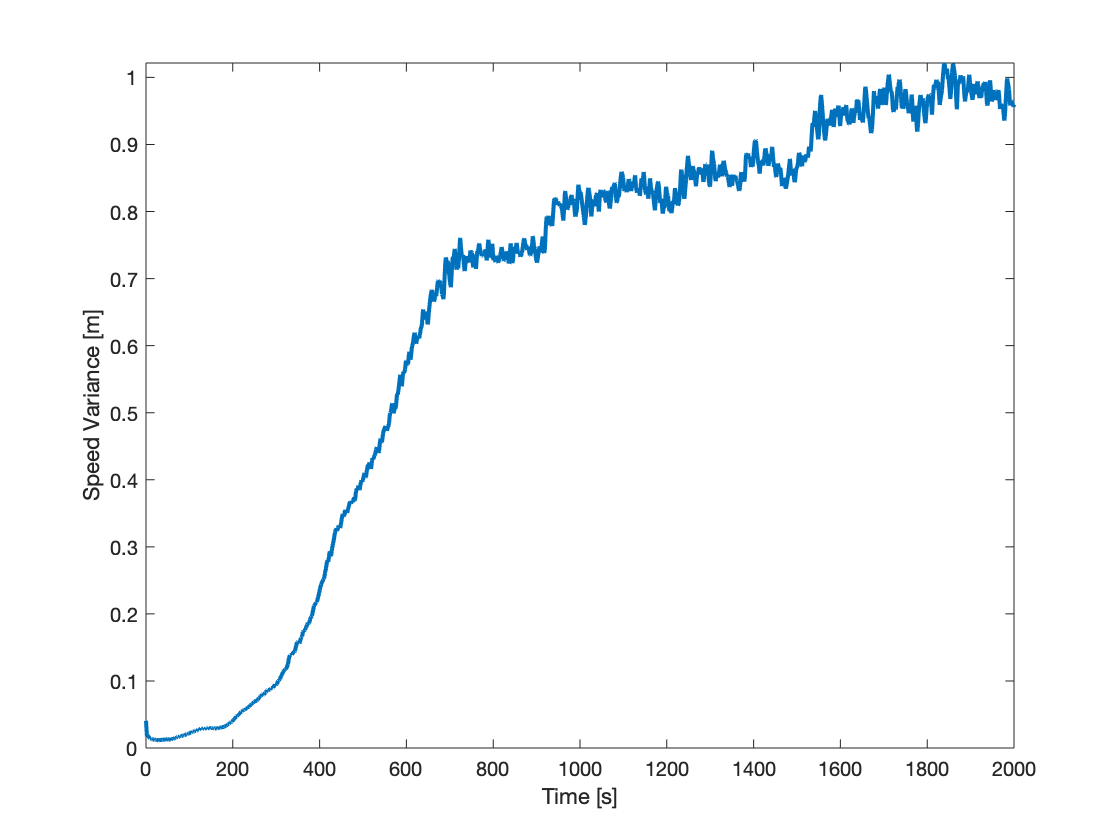}\includegraphics[width=0.48\textwidth]{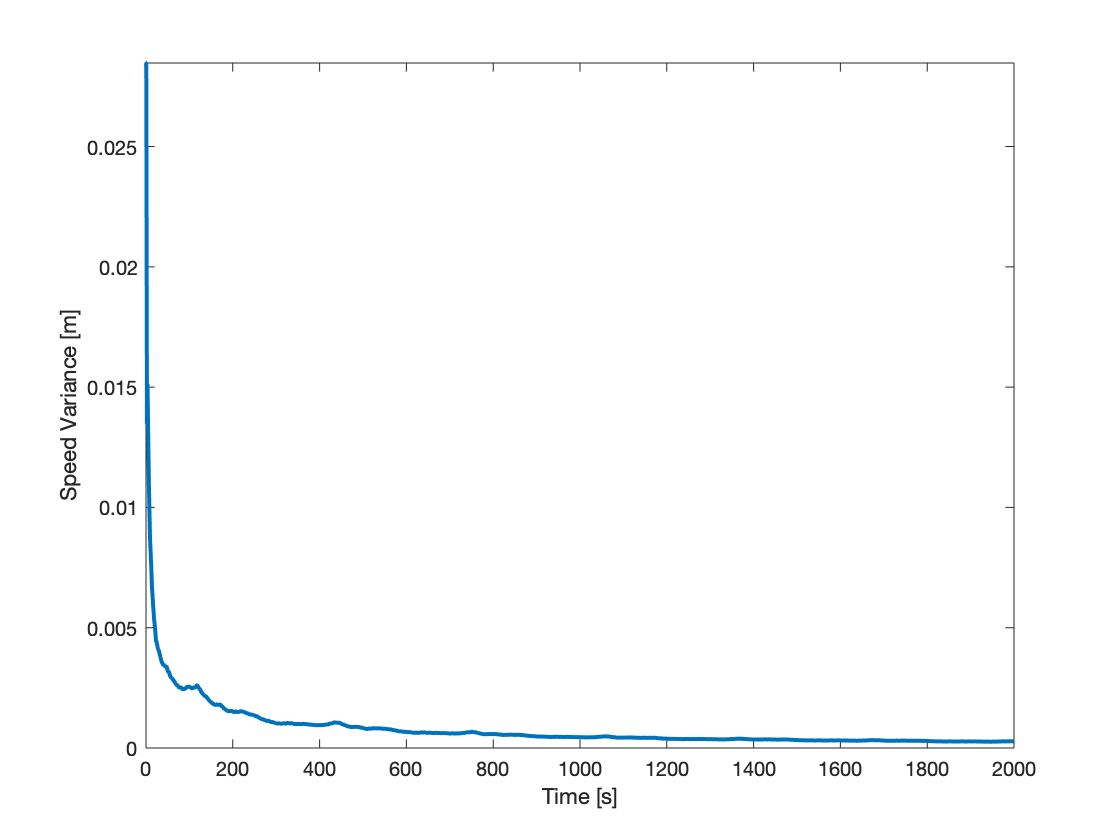}
    \caption{Speed variance over time in a two population traffic flow: drivers with stable behaviour ($\Delta_{1}=7.28$) and drivers with unstable behavior ($\Delta_{2}=-0.84$). Predicted critical penetration rate: $\tau_{0}=0.881$.
    \textbf{left:} $80\%$ drivers with stable behavior;   \textbf{right:} $88.2\%$ drivers with stable behavior.  \label{fig1}}
\end{figure}

%TODO ADD THIS PART !!!!
% However, as highlighted in Theorem \ref{th12} this instability might occur only with a large enough number of cars and might be missed when looking at experiments with a small number of cars such as \cite{Sugiyama, stern2018dissipation}rather than a real freeway. In Figure \ref{fig2} we show the terminal speed variance after 2000s of simulations with different values of $\tau$ and a number of cars ranging from 25 to 5000. In all these simulations we keep the same steady state described by $L/N = 10.4m$, hence we still have $\tau_{0}=0.881$. We see that the effective penetration rate above which the system is stable is lower than $\tau_{0}$ for a small number of cars and becomes very close for 5000 cars. 
% \begin{figure}
%     \centering
% %    \includegraphics{}
%     \caption{Terminal speed variance of a two population traffic flows with respect to proportion of aggressive drivers.\label{fig2}}
% \end{figure}
% %%%

\clearpage
 \bibliographystyle{plain}
 \bibliography{Biblio_traffic}
 \end{document}